\documentclass[11pt,graybox]{article}
\usepackage{amsmath,amsthm,amssymb}
\usepackage{url}
\usepackage{geometry}
\newtheorem{theorem}{Theorem}
\newtheorem{corollary}[theorem]{Corollary}
\newtheorem{lemma}[theorem]{Lemma}

\theoremstyle{definition}

\newtheorem{example}[theorem]{Example}

\usepackage{amssymb}
\usepackage{array,color,colortbl}
\providecommand{\mk}{\cellcolor[gray]{.70}}
\def\eref#1{$(\ref{#1})$}
\def\egref#1{Example~$\ref{#1}$}

\def\lref#1{Lemma~$\ref{#1}$}
\def\tref#1{Theorem~$\ref{#1}$}
\def\cref#1{Corollary~$\ref{#1}$}

\renewcommand{\geq}{\geqslant}
\renewcommand{\leq}{\leqslant}
\renewcommand{\ge}{\geqslant}
\renewcommand{\le}{\leqslant}
\def\<{\big\langle}
\def\>{\big\rangle}
\newcommand{\cset}{B}
\newcommand{\ncset}{\overline{\cset}}

\begin{document}

\title{\textbf{\Large{Trades in complex Hadamard matrices}}}

\author{
\textsc{Padraig \'O Cath\'ain}
				\thanks{\textit{E-mail: p.ocathain@gmail.com}}\\
\textsc{Ian M. Wanless}
				\thanks{\textit{E-mail: ian.wanless@monash.edu}}\\
\textit{\footnotesize{School of Mathematical Sciences,}}\\
\textit{\footnotesize{Monash University, VIC 3800, Australia.}}\\
}

\footnotetext{
This work was inspired by the discussion after Will Orrick's
talk at the ADTHM'14 workshop, and much of the work was undertaken
at the workshop. The authors are grateful to the workshop organisers
and to BIRS. Research supported by ARC grants FT110100065 and
DP120103067. This is the final form of this work.
No other version has been or will be submitted elsewhere.
}
\maketitle
\begin{center}
\begin{abstract}
\noindent
A {\em trade} in a complex Hadamard matrix is a set of entries which
  can be changed to obtain a different complex Hadamard matrix. We
  show that in a real Hadamard matrix of order $n$ all trades contain
  at least $n$ entries.  We call a trade {\em rectangular} if it
  consists of a submatrix that can be multiplied by some scalar
  $c\ne1$ to obtain another complex Hadamard matrix.  We give a
  characterisation of rectangular trades in complex Hadamard
  matrices of order $n$ and show that they all contain at least $n$
  entries. We conjecture that all trades in complex Hadamard matrices
  contain at least $n$ entries.
\end{abstract}

\end{center}

\vspace{0.3cm}

\noindent
{\bf 2010 Mathematics Subject classification:} 05B20, 15B34

\noindent
{\bf Keywords:} Hadamard matrix, trade, rank

\clearpage

\section{Introduction}

A \textit{complex Hadamard matrix} of order $n$ is an $n \times n$
complex matrix with unimodular entries which satisfies the matrix equation
\[ HH^{\dagger} = nI_{n},\] where $H^\dagger$ is the conjugate
transpose of $H$ and $I_n$ is the $n\times n$ identity matrix.  If the
entries are real (hence $\pm1$) the matrix is {\em Hadamard}. The
notion of a trade is well known in the study of $t$-designs and Latin
squares \cite{Billington}.  For a complex Hadamard matrix we define a
{\em trade} to be a set of entries which can be altered to obtain a
different complex Hadamard matrix of the same order.  In other words,
a set $T$ of entries in a complex Hadamard matrix $H$ is a trade if
there exists another complex Hadamard matrix $H'$ such that $H$ and
$H'$ disagree on every entry in $T$ but agree otherwise.  If $H$ is a
real Hadamard matrix, we insist that $H'$ is also real.

\begin{example}\label{eg:Pay8}
The 8 shaded entries in the Paley Hadamard matrix below form a trade.
\[
\left(
\begin{array}{cccccccc}
\mk+&\mk+&+&+&+&+&+&+\\
+&-&-&-&+&-&+&+\\
+&+&-&\mk-&\mk-&+&-&+\\
+&+&+&\mk-&\mk-&-&+&-\\
+&-&+&+&-&-&-&+\\
\mk+&\mk+&-&+&+&-&-&-\\
+&-&+&-&+&+&-&-\\
+&-&-&+&-&+&+&-\\
\end{array}
\right)
\]
If each of the shaded entries is replaced by its negative, the result is
another Hadamard matrix.
\end{example}

We use the word {\em switch} to describe the process of replacing a trade
by a new set of entries (which must themselves form a trade).
In keeping with the precedent from design theory, our trades
are simply a set of entries that can be switched.
Information about what they can be switched to
does not form part of the trade (although it may be helpful in
order to see that something is a trade). For real Hadamard matrices
there can only be one way to switch a given trade, since only two
symbols are allowed in the matrices and switching must change every
entry in a trade. However, for complex Hadamard matrices there can be
more than one way to switch a given trade, as our next example shows.

\begin{example}
  Let $u$ be a nontrivial third root of unity.  The following matrix
  is a $7 \times 7$ complex Hadamard matrix.  The shaded entries again
  form a trade; they can be multiplied by an arbitrary complex number
  $c$ of modulus $1$ to obtain another complex Hadamard matrix.
  This matrix is due originally to Petrescu \cite{Petrescu},
  and is available in the online database \cite{Karol}.
\def\sq{\rlap{$^2$}}
\def\phm{\phantom{-}}

\addtolength\arraycolsep{3pt}
\[\left(
\begin{array}{ccccccc}
1     & \phm 1          & \phm 1        & \phm 1    & \phm 1    & \phm 1   & \phm 1    \\
1     & \mk -u        & \mk \phm u        &-u\sq & -1    &-1   & -u    \\
1     & \mk \phm u         & \mk -u      & -1   & -u\sq    &-1   & -u    \\
1     & -u\sq    &-1       &\mk \phm u    &\mk -u    & -u   &-1    \\
1     &-1         &-u\sq   &\mk -u    &\mk \phm u    &-u   &-1    \\
1     &-1         &-1       &-u    &-u    &\phm u   &-u\sq    \\
1     &-u         &-u        &-1    &-1    &-u\sq   &\phm u    \\
\end{array} \right)
\]
\end{example}

The {\em size} of a trade is the number of entries in it.  We say that
a trade is {\em rectangular} if the entries in the trade form a
submatrix that can be switched by multiplying all entries in the trade
by some complex number $c\ne1$ of unit modulus.
It will follow from \lref{l:rect} that the value of $c$ is immaterial;
if one value works then they will all work.
In a complex Hadamard
matrix each row and column is a rectangular trade.
Thus there are always $1\times n$ and $n\times 1$ rectangular trades.
Similarly, we may exchange any pair of rows to obtain another
complex Hadamard matrix. In the real case,
the rows that we exchange necessarily differ in
exactly half the columns, so this reveals a $2\times\frac{n}{2}$
rectangular trade (and similarly there are always $\frac{n}{2}\times2$
rectangular trades in real Hadamard matrices). Less trivial trades
were used by Orrick \cite{Orr08} to generate many inequivalent
Hadamard matrices of orders 32 and 36. The smaller of Orrick's two
types of trades was a $4\times\frac{n}{4}$ rectangular trade that he
called a ``closed quadruple''. Closed quadruples are often but not
always present in Hadamard matrices. The trades just discussed all
have size equal to the order $n$ of the host matrix.  The trade in
\egref{eg:Pay8} is a non-rectangular example with the same property.

Trades in real Hadamard matrices and related codes and designs have
been studied occasionally  in the literature, either to produce
invariants to aid with classification or to produce many inequivalent
Hadamard matrices. See \cite{Orr08} and the references cited there. In
the complex case, trades are related to parameterising complex
Hadamard matrices, some computational and theoretical results are
surveyed in \cite{Tadej}.

Throughout this note we will assume that $H=[h_{ij}]$ is a complex
Hadamard matrix of order $n$. We will use $r_i$ and $c_j$ to denote the
$i$-th row and $j$-th column of $H$ respectively. If $B$ is a set of
columns then $r_{i,B}$ denotes the row vector which is equal to
$r_{i}$ on the coordinates $B$ and zero elsewhere. We use
$\overline{B}$ for the complement of the set $B$.


\section{Hadamard Trades}

We start with a basic property of trades. We use
$\<\cdot\,,\cdot\>$ for the standard Hermitian inner product under
which rows of a complex Hadamard matrix are orthogonal.

\begin{lemma}\label{l:rect}
  Let $T$ be a subset of the entries of a complex Hadamard matrix $H$.
  Let $c\ne1$ be a complex number of unit modulus.
\begin{enumerate}
\item
  Suppose that $T$ can be switched by multiplying its entries by $c$.
  Let $\cset$ be the set of columns in which
  row $r_{i}$ of $H$ contains elements of $T$.
  If $r_{j}$ is a row of $H$ that contains no
  elements of $T$ then $r_{i,\cset}$ is orthogonal to $r_{j,\cset}$.
\item
  Suppose that $T$ forms a rectangular submatrix of $H$ with rows $A$
  and columns $\cset$. Then $T$ can be switched by multiplying its
  entries by $c$ if and only if $r_{i,\cset}$ is orthogonal to
  $r_{j,\cset}$ for every $r_i\in A$ and $r_j\notin A$.
\end{enumerate}
\end{lemma}

\begin{proof}
First, since the rows of $H$ are orthogonal, we have that
\[ 0= \<r_{i},r_{j}\>
=\<r_{i,\cset},r_{j,\cset}\>+\<r_{i,\overline{\cset}},r_{j,\overline{\cset}}\>.
\]
Now, multiplying the entries in $T$ by $c$, we see that
\[ 0= \<cr_{i,\cset},r_{j,\cset}\>+\<r_{i,\overline{\cset}},r_{j,\overline{\cset}}\>
= c\<r_{i,\cset},r_{j,\cset}\>+\<r_{i,\overline{\cset}},r_{j,\overline{\cset}}\>. \]
Subtracting, we find that $(c-1)\<r_{i,\cset},r_{j,\cset}\> =0$.
Given that $c\neq1$ the first claim of the Lemma follows.

We have just shown the necessity of the condition in the second claim.
To check sufficiency we note that the above argument is reversible and
shows that $r_{i}\in A$ and $r_{j}\notin A$ will be orthogonal after
multiplication of the entries of $T$ by $c$.  So we just have to
verify that any two rows $r_{i}, r_{k}$ in $A$ will be orthogonal.
This follows from
\[
\<cr_{i,\cset},cr_{k,\cset}\>+\<r_{i,\overline{\cset}},r_{k,\overline{\cset}}\>
=|c|\<r_{i, \cset},r_{k,\cset}\>+\<r_{i,\overline{\cset}},r_{k,\overline{\cset}}\>
=\<r_{i, \cset},r_{k,\cset}\>+\<r_{i,\overline{\cset}},r_{k,\overline{\cset}}\>
= 0. \]
\vskip-10pt
\hfill\hfill\qedhere\end{proof}

Note that the value of $c$ plays no role in \lref{l:rect}. Also,
Part 1 of the lemma implies that in a real Hadamard matrix any
trade which does not intersect every row must use an even number
of entries from each row. The same is not true for trades in
complex Hadamard matrices (see \cite{Karol} for examples).

It is of interest to consider the size of a smallest possible
trade. For (real) Hadamard matrices of order $n$ we show that
arbitrary trades have size at least $n$. Equality is achievable in a
variety of ways, as discussed above. However, we find a restriction
that must be obeyed by any trade achieving equality.  Then we show
that in the general case rectangular trades have size at least
$n$. The question for arbitrary trades in complex Hadamard matrices
remains open.

\begin{theorem}\label{t:real}
Let $H$ be a (real) Hadamard matrix of order $n$.
Any trade in $H$ has size at least $n$. If $T$ is any trade of
size $n$ in $H$ then there are divisors $d$ and $e$ of $n$ such that
$T$ contains either $0$ or $d$ entries in each row of $H$ and
either $0$ or $e$ entries in each column of $H$.
Moreover, $d$ is even or $d=1$. Likewise, $e$ is even or $e=1$.
\end{theorem}

\begin{proof}
  Suppose that $H$ differs from a Hadamard matrix $H'$ in a trade $T$
  of at most $n$ entries.  Without loss of generality, we assume that
  $H$ is normalised, that the first row of $H$ contains
  $d$ differences between $H$ and $H'$, and
  that these differences occur in the first $d$ columns. We also assume
  that all differences between $H$ and $H'$ occur in the first $r$
  rows, with each of those rows having at least $d$ differences in them.
  The case $r = n$ is trivial, so we assume that $r < n$
  in the remainder of the proof.
  By assumption there are at least $rd$ entries in $T$, so $rd\le n$.
  Now consider the submatrix $S$ of $H$ formed by the first
  $d$ columns and the last $n-r$ rows. By \lref{l:rect}, we know that
  each row of $S$ is orthogonal to the all ones vector.  It follows
  that $d$ is even and $S$ contains $(n-r)d/2$ negative entries. The
  first column of $S$ consists entirely of ones so, by the pigeon-hole
  principle, some other column of $S$ must contain at least
\begin{equation}\label{e:negent}
\frac{(n-r)d}{2(d-1)} \le \frac{nd-n}{2(d-1)} = \frac{n}{2}
\end{equation}
  negative entries. This column of $H$ is orthogonal to the
  first column, so we must have equality in \eref{e:negent}.
  It follows that $n=rd$ and each of the first $r$ rows contain
  exactly $d$ entries in $T$. Columns have similar properties,
  by symmetry.
\hfill\hfill\qedhere\end{proof}

\begin{corollary}
  In a (real) Hadamard matrix of order $n$ the symmetric difference of any
  two trades must have size at least $n$.
\end{corollary}

\begin{proof}
Suppose that $H,H_1,H_2$ are distinct (real) Hadamard matrices of order $n$.
Let $T_1$ and $T_2$ be the set of entries of $H$ which disagree
with the corresponding entries of $H_1$ and $H_2$ respectively.
The symmetric difference of $T_1$ and $T_2$ has cardinality equal
to the number of entries of $H_1$ that are different to the
corresponding entry of $H_2$. This cardinality is at least $n$,
by \tref{t:real}.
\hfill\hfill\qedhere\end{proof}

\egref{eg:Pay8} is the symmetric difference of two rectangular trades,
one $2\times 4$ and the other $4\times 2$. It shows that equality can
be achieved in the Corollary.  The example also demonstrates that
trades of minimal size need not be rectangular. In the notation of
\tref{t:real} it has $d=e=2$ and $n=8$. Another example is obtained as
follows. Let $H$ be any Hadamard matrix and $H'$ the matrix obtained
by swapping two rows of $H$, then negating one of the rows that
was swapped. Let $T$ be the trade consisting of the entries of $H$
which differ from the corresponding entry in $H'$.  It is easy to
show that $T$ has $d=n/2$, $e=1$ in the notation of \tref{t:real}.

It is also possible to have $d=e=1$. If this is the case then by
permuting and/or negating rows we obtain a Hadamard matrix $H$ for
which $H-2I$ is also Hadamard, where $I$ is the identity matrix.
However this means that
\[
HH^\top=(H-2I)(H-2I)^\top=HH^\top-2H-2H^\top+4I.
\]
Hence $H+H^\top=2I$, so $H$ is a skew-Hadamard matrix. Conversely, the
main diagonal of any skew-Hadamard matrix is a trade with $d=e=1$.

Now we consider complex Hadamard matrices. The following lemma is the
key step in our proof.  The corresponding result for real Hadamard
matrices has been obtained by Alon (cf. \cite{Jukna}, Lemma 14.6).
Alon's proof can be trivially adapted to deal with complex Hadamard
matrices.  We include our own independent proof here since we want
to extract a characterisation of cases where the bound is tight.

\begin{lemma}\label{rowbound}
Let $H$ be a complex Hadamard matrix of order $n$,
and $\cset$ a set of $b$ columns of $H$.
If $\alpha$ is a non-zero linear combination of the elements of $\cset$ then
$\alpha$ has at least $\lceil \frac{n}{b}\rceil$ non-zero entries.
\end{lemma}

\begin{proof}
Without loss of generality, we can write $H$ in the form
\[ H = \left( \begin{array}{ll} T & U \\ V & W \end{array} \right) \]
where $T$ contains the columns in $\cset$ and the rows in which
$\alpha$ is non-zero.  We will identify a linear dependence among the
rows of $U$, then use this and an expression for the inner product of
$r_{1}$ and $r_{2}$ to derive the required result.  We assume that
there are $t$ non-zero entries $\alpha_i$ in $\alpha$ and that if $t\ge 2$ then
they obey
$|\alpha_{2}| \geq |\alpha_{1}| \geq |\alpha_{i}|$ for $3 \leq i \leq t$. We
need to show that $t \geq \lceil \frac{n}{b}\rceil$.

For any column $c_j$ not in $\cset$, we have that $\< c_{j}, \alpha\>
= 0$ since the columns of $H$ are orthogonal. Thus every column of $U$
is orthogonal to $\alpha$, and so there exists a linear dependence
among the rows of $U$, explicitly: $h_{1j} = \sum_{i = 2}^{t}
-\alpha_{i}\alpha_{1}^{-1} h_{ij}$, for any $j \notin \cset$. In
particular, this shows that indeed $t\ge2$.

Since $H$ is Hadamard, we know that all of the $h_{ij}$ have
absolute value $1$, and that rows of $H$ are necessarily orthogonal:
\begin{eqnarray*}
\< r_{1} , r_{2} \>
& =  & \< r_{1,\cset}, r_{2,\cset} \> + \< r_{1,\ncset} , r_{2,\ncset}\> \\
& =  & \< r_{1,\cset}, r_{2,\cset} \> +
     \< \sum_{i= 2}^{t} -\alpha_{i}\alpha_{1}^{-1} r_{i,\ncset}, r_{2,\ncset} \>\\
& = & \< r_{1,\cset}, r_{2,\cset} \> +
     \sum_{i= 2}^{t}-\alpha_{i}\alpha_{1}^{-1}\< r_{i,\ncset}, r_{2, \ncset} \>.
\end{eqnarray*}
Since $\<r_{1}, r_{2}\>=0$ and
$\< r_{i,\ncset}, r_{2, \ncset} \>=-\< r_{i,\cset}, r_{2, \cset} \>$,
this means that
\begin{equation}\label{e:alp}
\alpha_{2}\alpha_{1}^{-1} \< r_{2,\ncset}, r_{2,\ncset}\>
= \< r_{1,\cset}, r_{2,\cset} \> + \sum_{i = 3}^{t} \alpha_{i}\alpha_{1}^{-1}
    \< r_{i,\cset}, r_{2,\cset}\>.
\end{equation}
Now, each inner product $\< r_{i,B}, r_{2,B}\>$ is a sum
of $b$ complex numbers of modulus one, and
$|\alpha_{i}\alpha_{1}^{-1}| \leq 1$ for $i\ge3$. So the
absolute value of the right hand side of \eref{e:alp} is at most
$(t-1)b$. In contrast, the absolute value of the left hand side of
\eref{e:alp} is $|\alpha_{2}\alpha_{1}^{-1}|(n-b)\ge n-b$.
It follows that $n-b \leq (t-1)b$,
and hence $t \geq \lceil \frac{n}{b} \rceil$.
\hfill\hfill\qedhere\end{proof}

Let $H$ be a Fourier Hadamard matrix of order $n$, and suppose that
$t \mid n$.  Then there exist $t$ rows of $H$ containing only
$t^{\rm{th}}$ roots of unity.  Their sum vanishes on all but
$\frac{n}{t}$ coordinates, so \lref{rowbound} is best possible.
On the other hand, if $H$ is Fourier of prime order $p$, the only
vanishing sum of $p^{\rm{th}}$ roots is the complete one. So in
this case, a linear combination of at most $t$ rows will contain at
most $t$ zero entries.

\begin{theorem}\label{t:rect}
  If $H$ is a complex Hadamard matrix of order $n$ containing an $a\times b$
  rectangular trade $T$ then $ab \geq n$. If $ab = n$
  then $T$ is a rank one submatrix of~$H$.
\end{theorem}

\begin{proof}
Without loss of generality, $T$ lies in the first $a$ rows of $H$.
Let $B$ be the set of the columns that contain the entries of $T$.  By
hypothesis, $\gamma_{1} = \sum_{1\leq i\leq a} r_{i}$ and
$\gamma_{c}=\sum_{1\leq i\leq a}(c r_{i,\cset} + r_{i,\ncset})$ are both
orthogonal to the space $U$ spanned by the last $n-a$ rows of $H$. Now
consider $\gamma_{1} - \gamma_{c}$, which is zero in any column
outside $B$, but which is not zero since the rows of $H$ are linearly
independent. Observe that the orthogonal complement of $U$ is
$a$-dimensional, and that the initial $a$ rows of $H$ span this space:
thus $\gamma_{1} - \gamma_{c}$ is in the span of these rows,
\lref{rowbound} applies, and $ab \geq n$.

If $ab = n$ then, equality holds in calculations at the
end of the proof of \lref{rowbound}. In particular,
$|\langle r_{i,B},r_{2,B}\rangle|=b$ for each $i$, which implies
that $r_{i,B}$ is collinear to $r_{2,B}$.
Hence $T$ is a rank one submatrix of $H$.
\hfill\hfill\qedhere\end{proof}

We now give a complete characterisation of the minimal rectangular
trades in any complex Hadamard matrix.

\begin{theorem}
Let $H$ be a complex Hadamard matrix of order $n$ and
$T$ an $a\times b$ submatrix of $H$ with $ab = n$. Then $T$ is a
rectangular trade if and only if $T$ is rank~$1$.
\end{theorem}

\begin{proof}
\tref{t:rect} shows that any rectangular trade of size $n$ is
necessarily rank one.  So we need only prove the converse. Without
loss of generality, we assume that $T$ is contained in the first $a$
rows and first $b$ columns of $H$ and that $H$ is normalised.  Note
that this implies that $T$ is an all ones submatrix.

Consider $\gamma=(\gamma_1,\dots,\gamma_n) = \sum_{i= 1}^{b} c_{i}$,
the sum of the first $b$
columns of $H$.  It is clear that $\gamma_{j} = b$ for $j \in
\{1,\ldots,a\}$. If we show that $\gamma_{j} = 0$ for $a < j \le n$
then \lref{l:rect} will show that $T$ is a trade. We
calculate the $\ell_{2}$ norm of $\gamma$ in two ways: first, via an
expansion into orthogonal vectors:
\[ \|\gamma\|_{2}^{2}
= \langle \sum_{i = 1}^{b} c_{i}, \sum_{i = 1}^{b} c_{i}\rangle
= \sum_{i = 1}^{b}\langle c_{i}, c_{i}\rangle = bn. \]
On the other hand, $\|\gamma\|^{2}_{2} = \sum_{i= 1}^{n}|\gamma_{i}|^{2}$.
We have that $\gamma_{i} = b$ for $1\leq i\leq a$. But $ab^{2} = nb$,
so $|\gamma_{i}| = 0$ for all $i > a$.
Applying \lref{l:rect}, we are done.
\hfill\hfill\qedhere\end{proof}


\begin{corollary}
If $T$ is an $a\times b$ rank one submatrix of $H$, then $T$ is a trade if and only if
$ab = n$.
\end{corollary}

\begin{proof}
We have that $ab \geq n$ by Theorem \ref{t:rect}. In the other direction, Lindsay's Lemma states that
the size of a rank one submatrix of a Hadamard matrix of order $n$ is bounded above by $n$ (see Lemma 14.5 of \cite{Jukna}).
\hfill\hfill\qedhere\end{proof}

Ryser's embedding problem is to establish the minimal order, $R(a,b)$,
of a Hadamard matrix containing an $a\times b$ submatrix consisting
entirely of ones.  Any rank one submatrix can be transformed into a
submatrix consisting entirely of ones by a sequence of Hadamard
equivalence operations. Hence there is a Hadamard matrix of order
$ab$ containing an $a\times b$ rectangular trade if and only if
$R(a,b)=ab$.

Newman \cite{Newman} showed that $R(a,b)=ab$
whenever both $a,b$ are orders for which Hadamard matrices exist.
Michael \cite{Michael} showed that $R(a, b) \geq (a+1)b$ for odd
$a>1$.  Thus there are no $a\times b$
rectangular trades in this case, a
conclusion that could also be reached from \tref{t:real}.
Michael also showed that if $2a$ and $b/2$ are orders of
Hadamard matrices then there exists an $a\times b$ rectangular trade
in a Hadamard matrix of order $ab$. For example, there is a Hadamard
matrix of order $48$ containing a $6\times8$ rectangular trade.

\section{Open questions}

A Bush type Hadamard matrix of order $m^2$ contains an $m\times m$
rank one submatrix. Hence there is a Hadamard matrix of order $36$
containing a $6\times 6$ rectangular trade. Thus all cases of our first
question smaller than $a=6$, $b=10$ are resolved.

\smallskip\noindent
\textbf{Question 1:} Are there even integers $a,b$ for which
there does not exist a Hadamard matrix
of order $ab$ containing an $a\times b$ rectangular trade?

\medskip

On the basis of \tref{t:real} and \tref{t:rect} we are inclined to
think that the answer to the following question is negative:

\smallskip\noindent
\textbf{Question 2:} Can there exist trades of size less than $n$ in an
$n\times n$ complex Hadamard matrix?

\medskip

It would also be nice to know how ``universal'' the rectangular trades
we have studied are. \egref{eg:Pay8} showed that combinations of
rectangular trades can create more complicated trades. By iterating
such steps can we build all trades? In other words:

\smallskip\noindent
\textbf{Question 3:} Is every trade in a (real) Hadamard matrix a
$\mathbb{Z}_2$-linear combination of rectangular trades?
If so, how does this generalise to the complex case?

\medskip

This work was motivated in part by problems in the construction of
compressed sensing matrices \cite{DarrynOC}. Optimal complex
Hadamard matrices for this application have the property that linear
combinations of $t$ rows vanish in at most $t$ components.

\smallskip\noindent
\textbf{Question 4:} Other than Fourier matrices, are their
families of Hadamard matrices with
the property that no linear combination of $t$ rows contains more than
$t$ zeros? Or, if such matrices are rare, describe families in which
no linear combination of $t$ rows contain more than $f(t)$ zeros for
some slowly growing function $f$.

\medskip

We are indebted to Prof.~Robert Craigen for our final question
and the accompanying example.

\smallskip\noindent
\textbf{Question 5:} To what extent do the results in this paper
generalise to weighing matrices (and complex weighing matrices and
their generalisations)?  In particular, is the weight of a weighing
matrix a lower bound on the size of all trades in that matrix?
Note that any weighing matrix has a trade of size
equal to its weight, simply by negating a row.
Slightly less trivially, trades with size equal to the weight
can be obtained by weaving (see \cite{Craigen})
weighing matrices. For example,
take any $2\times 2$ block of rank one in the following
$\mathrm{W}(6,4)$. The shaded entries show one such block.
\[\left(
\begin{array}{rrrrrr}
0     & 0         & \mk +        & \mk +    & +    & +    \\
0     & 0         & \mk +        & \mk +    & -    & -    \\
+     & +         & 0            & 0        & +    & -    \\
+     & +         & 0            & 0        & -    & +    \\
+     & -         & +            & -        & 0    & 0    \\
-     & +         & -            & +        & 0    & 0    \\
\end{array} \right)
\]

\end{document}